\documentclass[12pt]{amsart}

\usepackage{amsmath} 
\usepackage{amssymb}

\theoremstyle{plain}
\newtheorem{theorem}{Theorem}[section]
\newtheorem{lemma}[theorem]{Lemma}
\newtheorem{proposition}[theorem]{Proposition}
\newtheorem{corollary}[theorem]{Corollary}

\theoremstyle{definition}
\newtheorem{definition}[theorem]{Definition}%[section]

\theoremstyle{remark}

\numberwithin{equation}{section}

\DeclareMathOperator{\fin}{fin}
\DeclareMathOperator{\triv}{triv}

\DeclareMathOperator{\dom}{dom}

\DeclareMathOperator{\ran}{range}
\DeclareMathOperator{\val}{val}

\newcommand{\tie}[1]
{\,\raise-5pt\hbox{
${\buildrel{\displaystyle{\rhd}\!\!{\lhd}}\over{
\scriptstyle
#1}}$}\,}

\newcommand{\card}[1]{\lvert #1 \rvert}

\newcommand{\forces}[2]{\Vdash_{#1} \mbox{``} #2 \mbox{''}}

\newcommand{\Poset}{{\mathbb P}}

\newcommand{\Naturals}{{\mathbb N}}

\newcommand{\pomega}{{\mathcal P}(\Naturals)}
\newcommand{\pomegaf}{{\mathcal P}(\Naturals)/\fin}

\title{
Non-trivial copies of $\mathbb N^*$
}

\author[A. Dow]{Alan Dow}
\address{Department of Mathematics and Statistics, UNC Charlotte}

\date{\today}

\begin{document}
\begin{abstract}
We show that it is consistent to have regular closed non-clopen
copies of
 $\mathbb N^*$ within $\mathbb N^*$ and a non-trivial self-map
of $\mathbb N^*$
 even if 
 all autohomeomorphisms of $\mathbb N^*$
are trivial. 
\end{abstract}
%\tableofcontents
\bibliographystyle{plain}

\maketitle

\section{Introduction} 

In this paper we are interested in the possible existence of regular closed
subsets  of $\mathbb N^*$ that are
non-trivial copies
of $\mathbb N^*$, i.e. that are not clopen.  A proper subspace $K\subset \mathbb
N^*$ is said to be a trivial copy if there is an embedding of $\beta
\mathbb N$ into $\beta \mathbb N$ which sends the remainder $\mathbb
N^*$  
onto $K$. More generally, a function  $h\in \mathbb N^{\mathbb N}$ is said to 
{\em induce\/} a function $F:\mathcal P(\mathbb N)
\rightarrow \mathcal P(\mathbb N)$ 
 on $I\subset \mathbb N$
  if $F(a) =^* h^{-1}(a) = \{ n\in \mathbb N:  h(n)  \in a\}$ for all $a\subset I$. 
The function $F$ 
is said to be {\em trivial\/} on $I$ if there is a such a 
function $h$. 
We use $\triv(F)$ to denote the ideal of sets
on which $F$ is trivial.
Such a function $F$ is usually a \textit{lifting\/} of a homomorphism
 $\psi :\mathcal P(\mathbb N)/\fin\rightarrow
  \mathcal P(\mathbb N)/\fin$ in the sense that
    $F(a)/\fin = \psi(a/\fin)$ for all
     $a\subset \mathbb N$.  The ideal $\triv(\psi)$ would coincide
     with that of $\triv(F)$ for any such lifting of $\psi$.

 It was shown in \cite{DowPAMS} that there are
non-trivial nowhere dense copies of $\mathbb N^*$ while
Farah \cite{Farah00} proved that PFA implies
that if $K\subset\mathbb N^*$ is homeomorphic
to $\mathbb N^*$, then there is a, possibly empty, clopen
subset $A$ of $\mathbb N^*$ such that $A\subset K$
and $K\setminus A$ is nowhere dense.

Velickovic \cite{veli.oca} introduced the poset   $\Poset_2$ 
which was created to produce a model of Martin's Axiom
in which $\mathfrak c=\aleph_2$ and in
which there are  non-trivial autohomeomorphisms on $\mathbb N^*$.
 This was achieved by forcing over a model of PFA.
Several variations of $\Poset_2$ are possible and we continue the 
study of the properties of $\mathbb N^*$ that hold in 
the model(s) obtained when forcing with $\Poset_2$ (and its variants) 
over a model of PFA (see also \cite{step.28,step.29,DSh2}).

Farah \cite{Farah00}
 defines the important
 notion of an ideal of $\mathcal P(\mathbb N)$ 
being ccc over fin to mean that
there is no uncountable almost disjoint family of subsets of $\mathbb
N$ none of which are in the ideal. By Stone duality, we define a
closed subset $K$ of $\mathbb N^*$ to be ccc over fin if there is
no  uncountable family of pairwise disjoint clopen subsets of
$\mathbb N^*$ each meeting $K$ in a non-empty set. It is shown in 
\cite{Farah00} that
PFA implies that for each homomorphism $\psi$ from $\mathcal P(\mathbb
N)/\fin$ onto $\mathcal P(\mathbb N)/\fin$, which is the Stone dual
of a  copy of $\mathbb N^*$ in $\mathbb N^*$,
  $\triv(\psi)$ is ccc
over fin. 
One of our main results is that  this is true of automorphisms $\psi$
 in the
forcing extension by $\Poset_2$ over a model of PFA.  
It was already shown in \cite{step.28} that 
$\triv(\psi)$ is a dense $P$-ideal in such models.

\section{Preliminaries }

Now we recall
 the  partial order $\Poset_2$ from \cite{veli.oca}.

\begin{definition}
    The\label{poset} 
partial order $\Poset_2$ is defined to consist of all 1-to-1
    functions 
$f$  where 
\begin{enumerate}
\item $\dom(f) = \ran(f)\subset \mathbb N$,
\item for all $i\in\dom(f)$ and $n\in\omega$, $f(i)\in
  [2^n,2^{n+1})$ if and only if 
 $i\in [2^n,2^{n+1})$
  \item $\limsup_{n\rightarrow\infty}\card{[2^n,2^{n+1})\setminus
\dom(f)} = \infty$\label{growth}
\item for all $i\in \dom(f)$, $i=f^2(i)\neq f(i)$.
\end{enumerate} 
The ordering on $\Poset_2$ is $\subseteq^{\ast}$.
\end{definition}

\begin{definition}
 The poset $\Poset_1$ is defined to consist of all $\{0,1\}$-valued
 partial functions $f$ such that 
 $\dom(f)  \subset \mathbb N$
 and   $\limsup_{n\rightarrow\infty}\card{[2^n,2^{n+1})\setminus
\dom(f)} = \infty$. 
The ordering on $\Poset_1$ is $\subseteq^{\ast}$.

The poset $\Poset_0$ is the subposet of $\Poset_1$ consisting
of those $f\in \Poset_1$ satisfying that
for all $n\in
\omega$, $f^{-1}(1)\cap [2^n, 2^{n+1})$ has size at most 1, and it
is non-empty if and only if  $[2^n, 2^{n+1})\subset \dom(f)$.
 
\end{definition}

Each of these posets introduces a new generic ultrafilter $\mathcal U$
which is a tie-point of $\mathbb N^*$ 
(as introduced in \cite{DTech1} $A \tie{\mathcal U}B$,
see also \cite{DSh1,DSh2}):
namely there is a cover by regular closed subsets $A,B$ satisfying
that $A\cap B = \{\mathcal U\}$.  For a regular closed set
$A$ of $\mathbb N^*$, we let $\mathcal I_A$ denote the ideal
 $\{ a\subset \mathbb N: a^* \subset A\}$.
  Let $G$ denote a
$\Poset_2$-generic filter. It is shown in \cite{veli.oca}, that the
collection $\mathcal U = \{ \mathbb N\setminus \dom(f) : f\in G\}$ is
an ultrafilter. This is also true for the posets $\Poset_0$ and $\Poset_1$.
In the cases of $G$ being generic for either of
$\Poset_0$ and $\Poset_1$, $\mathcal I_A $ would be
$\{ f^{-1}(1) : f\in G\}$, while, for $\Poset_2$, 
$\mathcal I_A = \{ \{ i\in \dom(f) : i < f(i)\} : f\in G\}$. 
This is discussed in \cite{DSh2}. One of our main motivations is to
discover if $A$ or $B$ can be homeomorphic to $\mathbb N^*$ as this
information can be quite useful in applications (again, see
\cite{DSh2}). 

If PFA holds, then each of $\Poset_0, \Poset_1, \Poset_2 $ is
$\aleph_1$-closed and $\aleph_2$-distributive (see
\cite[p.4226]{step.29}).  In this paper we will restrict our study to
forcing with these posets individually, but the reader is referred to
\cite{step.29} for the method to generalize to countable support
infinite products.  In particular, the result that $\triv(F)$ is ccc
over fin for all onto homomorphisms $F$ on $\mathcal P(\mathbb N)/\fin$
will hold in these more general models but we will only prove this for 
automorphisms in this paper.

\begin{theorem} If $G$ is $\Poset_0$-generic and $A\tie{\mathcal U}B$
  are as defined above, then there is a homeomorphism from 
  $A$ to $\mathbb N^*$ which extends to a continuous
  mapping with domain all of $\mathbb N^*$. 
\end{theorem}

\begin{proof} Let $\psi \in \mathbb N^{\mathbb N}$ be defined so that
 $\psi ([2^n,2^{n+1})) = \{n\}$ for all $n$, and let $\psi^*$ denote
the canonical extension with domain and range $\mathbb N^*$. In fact, 
 for each free ultrafilter $\mathcal W$, 
 the preimage
 of $\mathcal W$ under $\psi^*$ is the set of
 ultrafilters extending
 $\{ \psi^{-1}[W] : W\in \mathcal W\}$.
Recall that $A $ is the closure of the set
 $\bigcup\{ \left(f^{-1}(1)\right)^* : f\in G\}$. We will simply show
 that  $\psi^*\restriction A$ is one-to-one. 
Let $\mathcal V = \{ b \subset \mathbb N : \psi^{-1}(b) \in \mathcal
U\}$. 
By the definition of
 $\Poset_0$, it follows that, for each $f\in G$, 
 $\psi^* \restriction
 \left(f^{-1}(1)\right)^*$ is one-to-one  and that
 $\psi\left( f^{-1}(1) \right) \notin \mathcal V$. It  follows easily
 that the preimage of any point of $\mathbb N^*\setminus \{\mathcal
 V\}$ contains a single point in $A$. Now suppose that 
 $\mathcal W\neq \mathcal U$ is in the preimage of $\mathcal V$. Since
 $\mathcal U$ is generated by $\{ \mathbb N\setminus \dom(f) : f\in
 G\}$, we may choose an $f\in G$ with $\dom(f)\in \mathcal W$. Since
 $\psi\left( f^{-1}(1) \right) \notin \mathcal V$, we have that
 $f^{-1}(0)\in \mathcal W$. But now, $ f^{-1}(0) $ is mod finite
 disjoint from each member of $\mathcal I_A$, which shows that
 $\mathcal W\notin A$.
\end{proof}

The rest of the paper is devoted to proving the following theorems. 

\begin{theorem} 
In the extension obtained by forcing over a model of PFA by any
 of $\Poset_0$, $\Poset_1$, or $\Poset_2$,
 if $\Phi$ is an\label{cccfin} 
automorphism of $\pomegaf$,
then $\triv(\Phi)$ is a ccc over fin ideal. 
\end{theorem}

\begin{theorem}
In the extension obtained by forcing over a model of PFA
by  $\Poset_0$\label{noauto}  all
 automorphisms on $\pomegaf$ are trivial.
\end{theorem}

\begin{theorem}
In the extension obtained by forcing over a model of PFA
by $\Poset_1$, there\label{noN*}
are non-trivial automorphisms.
\end{theorem}

These results are re-stated and proven as
theorems \ref{repeatcccfin}, \ref{repeatnoauto}, and \ref{ramsey} respectively. 
Let us again remark that   Theorem \ref{cccfin} holds
for all onto homomorphisms but this requires lengthy verifications
that the results for automorphisms from \cite{step.28,step.29,DSh2} 
also hold for onto homomorphisms.
As mentioned above a homomorphism $\psi $ from $\mathcal P(\mathbb N)/\fin $ 
onto $\mathcal
P(\mathbb N)/\fin$ is said to be trivial, if there is function $h
\in \mathbb N^{\mathbb N}$ which induces $\psi$ in
the sense that $\psi(a/\fin) = \left(h^{-1}(a)\right)/\fin$
 for all $a\subset \mathbb N$.
We intend to  deal with automorphisms only so 
it will be  more convenient  to work with the inverse map; hence
 $\psi(a/fin) = h(a)/\fin$. 
We will say that $\psi $ is not trivial at a point
$x\in \mathbb N^*$ if no member of $\mbox{triv}(\psi)$ is in the ultrafilter
corresponding to $x$.

\section{The auxillary poset $\Poset(\mathfrak F)$}

For  this section
 let $\Poset$ denote any one of the posets
 $\mathbb P_0, \mathbb P_1, \mathbb P_2$.  It is known that
 $\Poset$ is $\sigma$-directed closed.
 The following partial order was introduced in \cite{step.28} as a  
tool to uncover the forcing preservation properties of $\Poset$, such
as Velickovic's result that PFA implies that 
 $\Poset$ is $\aleph_2$-distributive
(and so introduces no new $\omega_1$-sequences of subsets of $\mathbb
N$).

\begin{definition}
 Let $\mathfrak F$ denote any filter
on $\Poset$. Define
   $\Poset(\mathfrak{F})$  
to be the partial order consisting of all $g\in
\Poset$ 
such that there is some $f\in \mathfrak{F}$ which is almost equal to
it.  The ordering on  $\Poset(\mathfrak{F})$ 
is $f\leq g$ if $f\supseteq g$.
\end{definition}

The forcing poset
 $\Poset(\mathfrak  F)$  (which is just the set $\mathcal F$)
  introduces a new total
function $f$ which 
extends mod finite 
 every member of $\mathfrak F$.
 Although $f$ will  not be a member of
$\Poset$ it is only 
 because its domain does not satisfy the growth condition
 (\ref{growth}) in the definition of $\Poset$.
There is a simple
 $\sigma$-centered poset $\mathcal S$ which will force
 an appropriate set $I\subset\mathbb N$ which
 mod finite contains all the domains of members of $\mathfrak F$
and satisfies that $f\restriction I$ is a member of $\Poset$ 
which is below each member of $\mathfrak  F$
  (see \cite[2.1]{step.28}).

A strategic choice of the filter $\mathfrak F$ will ensure that
$\Poset(\mathfrak F)$ is ccc and much more.
Again we are lifting results from \cite[2.6]{step.28} and
\cite[proof of Thm. 3.1]{step.29}. 
  A poset is said to be $\omega^\omega$-bounding
if every new function in $\omega^\omega$ is bounded by some ground
model function. 

\begin{lemma}[\cite{step.29}]
In the forcing extension, $V[H]$,
 by $2^{<\omega_1}$,
  there is a maximal filter
$\mathfrak F$ on  $\Poset$ which
  is $\Poset$-generic over $V$ and for which 
$\Poset(\mathfrak F)$ is ccc, $\omega^\omega$-bounding, 
and preserves that $\mathbb R\cap V$ is not meager.
\end{lemma}

Almost all of the work we have to do is to establish additional
preservation results for the poset(s) $\Poset(\mathfrak F)$. Once
these are established, we are able to apply the standard PFA type
methodology as demonstrated in \cite{step.28, step.29}
to determine properties of the forcing extension by $\Poset$.
As mentioned above, we have this result from \cite[Corollary 3.3]{step.29}.

\begin{lemma}In the model $V[H]$,
   the ideal $\triv(F)$ is a\label{Pideal} dense P-ideal.
\end{lemma}

\section{$\sigma$-Borel liftings and ccc over fin}

A lifting of a map $\Phi$ 
 from $\mathcal P(\mathbb N)/\fin$ to 
itself is any function $F$
 from $\mathcal P(\mathbb N)$ into $\mathcal
P(\mathbb N)$ which satisfies that $F(a)/\fin = \Phi(a/\fin)$ for all
$a\in \mathcal P(\mathbb N)$. For each $\ell\in \mathbb N$ and
$s\subset \ell$, let $[s;\ell] = \{ x\subset \mathbb N : x\cap \ell =
s\}$. This defines the standard Polish topology on $\mathcal P(\mathbb
N)$.
For a set $\mathcal C\subset \pomega$ and a function $F$
 on $\pomega$, let
 us say that $F\restriction \mathcal C$ is
 $\sigma$-Borel  if there is sequence $\{\psi_n : n\in
 \omega\}$ of Borel functions on $\pomega$ such that for each 
$b\in \mathcal C$, there is an $n$ such that $F(b)=^* \psi_n(b)$.

We continue the analysis of $\Poset$-names from $V$ in 
the forcing extension $V[H]$ using a $V$-generic filter $\mathcal
F\subset \Poset$  (continuing that $\Poset$ is one of
 $\Poset_0,\Poset_1, \Poset_2$).
 In particular, fix $\dot \Phi$  a
$\Poset$-name which is forced by $\mathbf{1}_{\Poset}$
 to be a lifting of an
automorphism  of   $\mathcal
P(\mathbb N)/\fin$. Let 
 $F$ denote $\val_{\mathcal F}(\dot \Phi)$. Of course it follows that,
 in $V[H]$,
 $F$ is a lifting of an automorphism of   $\mathcal
P(\mathbb N)/\fin$. Note that forcing with 
$2^{<\omega_1}$ does not change the set
 $\mathcal P(\mathbb N)$.

The following key result of  (\cite[2.3]{DSh2})  was extracted from 
\cite{step.28} and \cite[Theorem 3.3]{step.29}.  

\begin{lemma}[PFA]
For any\label{mainlemma}  dense $P$-ideal $\mathcal I$ on 
$\Naturals$ and
for each $\Poset(\mathfrak F)$-generic filter
$G$,  there is an $I\in \mathcal I $ such 
that $F\restriction \left(V\cap [\Naturals\setminus 
I]^\omega\right)$ is $\sigma$-Borel in the extension
 $V[H][G]$.
\end{lemma}

One of the main results which we can extract from  \cite{step.29} 
and simply deduce from Lemma \ref{Pideal} and Lemma \ref{mainlemma}
is
the following.

\begin{lemma} $F\restriction (V\cap \mathcal P(\mathbb N))$ is
  $\sigma$-Borel in the extension\label{sigmaBorel}
  obtained by forcing with
  $\Poset(\mathfrak F)$. 
\end{lemma}

We will also need several results from \cite{DSh2}. The following
 are \cite[3.1]{DSh2}  and  \cite[2.5]{DSh2} respectively.

\begin{lemma} 
Assume\label{most-useful-lemma}  
that $b\in V\cap [\Naturals]^\omega$ is such that
$F\restriction [V\cap [b]^\omega]$ is $\sigma$-Borel in $V[G]$. Then,
  in $V$, there is an 
increasing sequence $\{ n_k  : k \in \omega\}\subset \omega$ 
such that $F$ is trivial 
on each $a\in [b]^\omega$ for which there is an $r\in
  \mathfrak F$, such that $a\subset \bigcup \{
 [n_k,n_{k+1}) : 
[n_k,n_{k+1})\subset  \dom(r)\}$.
\end{lemma}

\begin{lemma} 
Let\label{also-useful-lemma}   $H$ and  $\mathfrak F$ be as above.
Then for each $\Poset(\mathfrak
  F)$-name $\dot h \in \Naturals^\Naturals$ there are
an increasing sequence $n_0<n_1<\cdots$ of integers and a condition
$  f\in \mathfrak F$  such that 
either
\begin{enumerate}
\item  
$ f\forces{\Poset{(\mathfrak F)}}{\dot 
  h\restriction \bigcup \{ [n_k,n_{k+1}) : k\in K\}
 \notin V}$ for each infinite $K\subset \omega$
or 
\item
for each $i\in [n_k,n_{k+1})$
and  each $g<  f$ such that $g$ forces a\label{two.preserve1} 
value on $\dot h(i)$, 
$  f  \cup (g \restriction[n_k,n_{k+1}))$ also forces
 a value on $\dot h(i)$.
\end{enumerate}
Furthermore, if $f$ forces $\dot h$ to be finite-to-one, we can 
arrange that for each $k$ and each $i\in [n_k,n_{k+1})$,
$  f$ forces that $\dot h(i)\in
  [n_{k-1},n_{k+2})$. 
\end{lemma}

\bigskip

Next we need to  use a key  Lemma from \cite[Lemma 3.1]{DSh2}.

\begin{lemma} There is a condition $\bar p\in \mathfrak F$ and an\label{mainlemma}
increasing sequence $\{ n_k  : k \in \omega\}\subset \mathbb N$ 
such that $\bar p$ forces (over $\Poset$) that  $\triv(F) $ contains all 
$a\subset \mathbb N$ for which there is an  
 $r
\in
  \mathfrak F$, such that $a\subset \bigcup \{
 [n_k,n_{k+1}) : 
[n_k,n_{k+1})\subset  \dom(r)\}$.
\end{lemma}

We are now ready to complete the proof 
of Theorem \ref{cccfin}. For the remainder of this
section we assume $\bar p$ and $\langle n_k : k\in\omega\rangle$
have the properties in Lemma \ref{mainlemma}.
We proceed by contradiction using the following Lemma.

\begin{lemma} If $p$ forces that $\triv(\dot\Phi)$ is not ccc over fin,
then there is $\bar p<p$ and\label{notcccoverfin}
almost disjoint families
 $\{ a_\alpha : \alpha <\omega_1\}\subset [\mathbb N]^\omega$ 
 and
 $\{ b_\alpha : \alpha <\omega_1\}\subset [\mathbb N]^\omega$  
 such that
 $\{ a_\alpha\cup b_\alpha : \alpha<\omega_1\}$ is also
 an almost disjoint family, and 
  $\bar p$ forces that $a_\alpha\notin\triv(F)$
  and 
 $F(a_\alpha)=b_\alpha$
 for each $\alpha < \omega_1$.
 \end{lemma}

\begin{proof}
 Using that $\Poset$ is
$\aleph_2$-distributive we then have that $\triv(F)$ is not ccc over
fin.  Also, we can assume that $p$ forces the following.
The almost disjoint family $\{a_\alpha: \alpha <\omega_1\}
\subset [\mathbb N]^{\omega}$  satisfies
that $a_\alpha\notin\triv(F)$ 
and that $F(a_\alpha)=b_\alpha$ for all
$\alpha\in\omega_1$. 
Notice that the family
 $\{ b_\alpha : \alpha\in\omega_1\}$
 is also an almost disjoint family.
 By compactness of $a_\alpha^*$, we may choose an ultrafilter
  $\mathcal W_\alpha$ on $\mathbb N$ so 
  that $a_\alpha\in \mathcal W_\alpha$ and so
  that   $F$ is not trivial at $\mathcal W_\alpha$. 
  If there is an uncountable set of $\alpha$ such
  that $F(\mathcal W_\alpha)=\mathcal W_\alpha$,
   then we pass to such an uncountable subcollection
as well as shrink each $a_\alpha$ so that 
the new $b_\alpha$ is a subset of the original $a_\alpha$. 
 
 So we now assume that $F(\mathcal W_\alpha)\neq \mathcal W_\alpha$
 for all $\alpha\in\omega_1$. 
 Next, for each $\gamma\in\omega_1$, let 
   $S_\gamma$ be the set of all $\alpha\in \omega_1$
   such that either $a_\gamma\in F(\mathcal W_\alpha)$ or
    $b_\gamma \in \mathcal W_\alpha$. If, there is some 
     $\gamma\in \omega_1$ such that $S_\gamma$
     is uncountable, then we can, by shrinking
      $a_\alpha$ for all $\alpha\in S_\gamma$
      ensure that either $a_\alpha\subset F(a_\gamma)$
      or $b_\alpha\subset a_\gamma$. In either
      case, we obtain a new family, 
      $a_\alpha\in \mathcal W_\alpha$
      for $\alpha\in S_\gamma$,
      such that 
        $\{  a_\alpha\cup b_\alpha : \alpha \in S_\gamma\}$
        is an almost disjoint family.        
        Otherwise, we have that $S_\gamma$ is countable for
        every $\gamma\in \omega_1$. Recursively choose
        an uncountable subcollection $\{ \alpha_\xi : \xi\in\omega_1\}
        \subset\omega_1$ so that 
         $\alpha_\xi \notin S_{\alpha_\eta}$ for all $\eta<\xi$.        
First suppose there is a $\delta<\omega_1$ such that
there is an uncountable set $S\subset\omega_1$
with each $\mathcal W_{\alpha_\xi}$ (for $\xi\in S$)
 being in the $\beta\mathbb N$ closure of the union
  of the family of clopen sets $\{ b_{\alpha_\eta}^*  : \eta<\delta\}$. 
            It then follows that, for all $\xi\in S$, $\mathcal W_{\alpha_\xi}$
            is not in the closure of the family
            $\{ b_{\alpha_\eta}^* : \delta\leq \eta <\xi\}$. 
            For each $\xi\in S\setminus\delta$, 
            replace $a_{\alpha_\xi}$ be a smaller
             $a_{\alpha_\xi}\in \mathcal W_{\alpha_\xi}$
             satisfying that $a_{\alpha_\xi}\cap b_{\alpha_\eta}$
             is finite for all $\delta\leq \eta<\xi$. 
           If there is no such $\delta<\omega_1$, 
           then we can assume, by passing to an uncountable subfamily,
           that $\mathcal W_{\alpha_\xi}$ is not in the closure
           of the union of the family
           $\{ b_{\alpha_\eta}^*  : \eta<\xi\}$. In this case we
           may also assume that each $a_{\alpha_\xi}\in
            \mathcal W_{\alpha_\xi}$ is mod finite disjoint
            from $b_{\alpha_\eta}$ for all $\eta<\xi$.                       
                      By symmetry, we may perform the same
                      reduction so that for all $\xi<\omega_1$,
                       $b_{\alpha_\xi}$ is almost disjoint
                       from $a_{\alpha_\eta}$ for all
                        $\eta<\xi$.  
                        
                        We have now shown that the family
                         $\{ a_{\alpha}\cup b_\alpha : \alpha < \omega_1\}$
                         can be assumed to be almost disjoint.  
                         It is not needed but it is interesting that
                         since this argument can take place
                         in the ground model of PFA,
                          we have, by a result of Shelah reported
                          in \cite[3.11]{Dowomega1}, 
                         that we may assume that 
                         either $F(\mathcal W_\alpha)=\mathcal W_\alpha$
                         for all $\alpha\in\omega_1$ or
                         the sets $\{ \mathcal W_\alpha : \alpha \in\omega_1\}$
                         and $\{ F(\mathcal W_\alpha : \alpha\in \omega_1\}$
                         have disjoint closures in $\beta\mathbb N$.
  \end{proof}
  
  Finally, 
by 
possibly thinning out the collection, with
 $F(a_\alpha)  = b_\alpha$, assume that $b_\alpha$ is 
not a member of $\mathcal W_\beta$ for all $\beta \neq \alpha$.
Fix any $p\in \mathcal F$ which forces,
over $\Poset$,  that $\dot \Phi$ and, therefore $F$ have
all of the
above properties.
Assume also that $p$ satisfies the requirement
in Lemma \ref{mainlemma}.

Notice that for each $q\in \mathfrak F$, we have a 
one-to-one function $h_q$ with
domain 
 $a_q = \bigcup \{ [n_k,n_{k+1} ) : [n_k,n_{k+1})\subset \dom(q)\}$
which witnesses that  
$a_q\in \triv(F)$. 
Therefore the family 
 $\{ h_q : q\in \mathfrak F\}$ is a $\sigma$-directed (mod finite) 
family of functions 
 which, because $\mathbb N\notin\triv(F)$ in $V[H]$,
  has no extension in $V[H]$.

\begin{lemma} The family $\{ \dom(h_q) : q\in \mathfrak F\}$ generates 
  a dense ideal in $V[H]$ which remains dense after forcing with
   $\Poset(\mathfrak F)$.
\end{lemma}

\begin{proof} The  finite-to-one map sending $[n_k,n_{k+1})$ to $k$ is
    easily seen to send the family $\{ a_q  : q\in \mathfrak F\}$
    to a maximal ideal, 
and it is also the preimage of this maximal ideal.
 The forcing $\mathbb P(\mathfrak F)$ is 
$\omega^\omega$-bounding and so does not diagonalize the dual
    ultrafilter. 
\end{proof}
 
 We prove,
 using this next well-known result, that
 there is a $\Poset(\mathfrak F)$-name,   $\dot h$, of
  a function in $\mathbb N^{\mathbb N}$ that is forced 
  to mod finite extend every member
  of $\{h_q : q\in \mathfrak F\}$.  
Uncountable pairwise incompatible families
of partial functions on $\mathbb N$ are
similar to Luzin gaps.
Such a family will not have a
common mod finite extension.  The following result is an easy
consequence of a result of 
Todorcevic (see  \cite[2.2.1]{Farah00}
and \cite{DoSiVa}). Proper posets were introduced
in \cite{shelah.pf}.

\begin{proposition} If\label{functiongaps} 
$\{ h_\alpha : \alpha \in \omega_1\}$ is a family
of 
partial functions on $\mathbb N$ with mod finite increasing domains,
and if  there is no
common mod finite extension, then there is a proper poset 
 which introduces an uncountable pairwise incompatible subfamily.
\end{proposition}

 So now we assume that forcing with $\Poset(\mathfrak F)$
 preserves that the family $\{ h_q : q\in \mathfrak F\}$ has
 no common mod finite extension.  Let $\dot q_{\omega_1}$
 denote the $\Poset(\mathfrak F)$-name of the function
 that equals the union of the generic filter $G$.  
 We note that the sequence
   $\{ n_k : k\in\omega\}$ is a sequence from $V$.  
  Let $\dot Q$ be
 the $\Poset(\mathfrak F)$-name of the proper poset 
 as described in Lemma \ref{functiongaps}. That is,
 there is a $2^{<\omega_1}*\Poset(\mathfrak F)*\dot Q$-name,
   $\{ (\dot q_\alpha, 
   \dot a_\alpha , \dot h_\alpha,\dot L_\alpha) : \alpha\in \omega_1\}$,
   satisfying that, for each $\alpha < \beta \in \omega_1$, 
   it is forced that 
   \begin{enumerate}
    \item $\{  
    \dot q_\alpha, 
   \dot a_\alpha , \dot h_\alpha
    : \alpha \in \omega_1\}$ are  $\Poset(\mathcal F)$-names,
    \item 
    $\dot q_\alpha \subset \dot q_{\omega_1}  $,
  
    \item $\dot a_\alpha =\bigcup\{ [n_k, n_{k+1}) 
      : k\in \dot L_\alpha\}\subset \dom(\dot q_\alpha)$,
       \item $\dot q_\alpha$ forces that $\dot h_\alpha \in V$ induces 
       $\dot \Phi$ on $\dot a_\alpha$, 
      \item $\dot L_\alpha \subset^* \dot L_\beta$, 
   \item 
    there is an $n\in\dot a_\alpha\cap \dot a_\beta$ with 
    $\dot h_\alpha(n)\neq \dot h_\beta(n)$. 
   \end{enumerate}
   
Next let $\dot R_1$ be the
$2^{<\omega_1}*\Poset(\mathfrak F)*\dot Q$-name of
the
 usual $\sigma$-centered poset that will force the
 mod finite ascending sequence $\{ \dot L_\alpha : \alpha\in\omega_1\}$
  has a co-infinite extension.  That is, 
   there is a 
   $2^{<\omega_1}*\Poset(\mathfrak F)*\dot Q*\dot R_1$-name,
    $\dot L$, such that it is forced that $\mathbb N\setminus \dot L$
    is infinite and $\dot L_\alpha \subset^* \dot L$ for all
    $\alpha < \omega_1$. 
    And finally, let $\dot R_2$ be
    the
$2^{<\omega_1}*\Poset(\mathfrak F)*\dot Q*\dot R_1$-name
of the $\sigma$-centered poset such that there is a name,
evidently a  $2^{<\omega_1}*\Poset(\mathfrak F)*\dot Q*\dot R_1
*\dot R_2$-name, $\dot S$ of a subset of $\mathbb N$
such that it is forced
that  
$\dot q_{\omega_1}\restriction \dot S \in \Poset$, 
 $\bigcup\{[n_k,n_{k+1}) : k\in \dot L\} \subset \dot S$,
  $\dom(\dot q_\alpha)\subset^* \dot S$ for all
  $\alpha<\omega_1$. 
  
  Now return to the PFA model and choose a filter $\Gamma$ 
  on   
$2^{<\omega_1}*\Poset(\mathfrak F)*\dot Q*\dot R_1
*\dot R_2$ that meets $\omega_1$-many dense open sets
sufficient to ensure that properties
 (1)-(5) of the valuations of the names
  $\{ (\dot q_\alpha, \dot a_\alpha,\dot h_\alpha, \dot L_\alpha) : 
   \alpha<\omega_1\}$ will all hold
   and so that  the valuations of 
   $\dot L$, $\dot q_{\omega_1}$, and $\dot S$
   all have the properties mentioned above.  
   We may also assume that the condition $p$ of Lemma \ref{mainlemma}
   is an element of $\Gamma$ in the sense that
    the condition $(1,p,1,1,1)$ (as an element of
     $2^{<\omega_1}*\Poset(\mathfrak F)*\dot Q*\dot R_1
*\dot R_2$) is an element of $\Gamma$. Now let
 $\{ h_\alpha :\alpha \in \omega_1\}$ be the valuations
 of the elements of $\{ \dot h_\alpha : \alpha\in\omega_1\}$
 by the filter $\Gamma$. Also let $L$ be the valuation
 of $\dot L$ by $\Gamma$ and let
  $r$ be the valuation 
 of the condition $\dot q_{\omega_1}\restriction \dot S$
 by $\Gamma$. We observe that $a_r = \bigcup \{[n_k,n_{k+1}) : 
  k\in L\}$ is a subset of $\dom (r)$ and so, by Lemma \ref{mainlemma}
   there
  is a function $h_r$ (with domain $a_r$)
  that induces $F$ on $\mathcal P( a_r)$.  Choose an uncountable
  $\Lambda\subset\omega_1$ so that there is a $\bar k\in\omega$
  such that $L_\alpha\setminus L\subset \bar k$ and
   $h_\alpha \restriction \left(a_\alpha\setminus n_{\bar k}\right)
   \subset h_r$
 for all
   $\alpha\in \Lambda$. We may further suppose
   that $h_\alpha\restriction (a_\alpha\cap n_{\bar k})
    = h_\beta\restriction (a_\beta\cap n_{\bar k})$
    for all $\alpha,\beta\in \Lambda$. 
    It should now be clear that for all $\alpha,\beta\in \Lambda$,
     condition (6) does not hold.

  This contradiction therefore shows that there is a 
   $\mathbb P(\mathfrak
F)$-name  $\dot h$ 
  of a function on $\mathbb N$   that is forced
  to mod finite extend every member
  of $\{ h_q : q\in \mathfrak F\}$.

 For partial functions
 $p$ and $s$ with domains contained in $\mathbb N$, we let $s\sqcup p$
 denote the function $
s \cup (p\restriction (\dom(p)\setminus \dom(s)))$.  Since
 $\dot h$ is forced to mod finite extend each $h_q$ (for $q\in \mathfrak F$),
  we have that condition (1) of Lemma \ref{also-useful-lemma} fails to hold.
Furthermore, we can suppose that there is a $\bar p\supset \bar q\in \mathfrak F$
 such that 
 for each $k$, there is a
single $m_k$ such that $[2^{m_k},2^{m_k+1})$ properly contains
$ [n_k,n_{k+1})\setminus \dom(\bar q)$.    
     By further grouping and by
 extending the condition $\bar q $ we can assume that for all $k$ and
 $j\in [n_k, n_{k+1})$, if $\bar q $ does not force a value on $\dot
 h(j)$,
 then $\bar q $ does force that $\dot h(j) \in [n_k, n_{k+1})$.  
  We record this as a new Lemma.

\begin{lemma} 
It is forced by $\bar q<\bar p$ that the increasing sequence 
$n_0<n_1<\cdots$ of integers\label{preserve2}
 satisfies  that for each $k\in \mathbb N$ and
for each $i\in [n_k,n_{k+1})$ such 
that $\bar q$ does not force a value on $\dot h(i)$
and for each $q<\bar q$ that does force a
value on $\dot h(i)$, 
$ (q\restriction[n_k,n_{k+1}))\sqcup \bar q$ also forces
 that  value on $\dot h(i)$ and that value is in $
 [n_{k},n_{k+1})$. 
\end{lemma}
 
For each $k$, let  $
H_k = [n_k,n_{k+1})\setminus \dom(\bar q ) =
 [2^{m_k},2^{m_k+1})\setminus \dom(\bar q )$. 
 Finally, 
let  $\mathcal H_k $ 
 denote the set of functions  $s$ with domain contained in $H_k$
  for which  there is a $q \leq  \bar q $ with $s=q\restriction H_k$. 
   Recall that 
for the posets $\Poset_0$ and $\Poset_1$ the conditions are functions into
  $2$, while for the poset $\Poset_2$, the conditions $q$ extending $\bar q $
  are permutations which
  send each $H_k$ into itself. 
Therefore, with $\Poset$ being any of the three posets considered in this paper,
 $\mathcal H_k$ is a finite set of functions with domain and
  range contained in $2\cup H_k$.  
 For the remainder of the
 section we will assume that each condition we choose in $\mathbb
 P(\mathfrak F) $ is below this $\bar q $.

\begin{lemma} If $Y = \{ y_k : k\in \omega\}$ 
 is\label{grabY}
 such that $y_k \in
 [n_k,n_{k+1})$ for each $k$, then for each $
 q\in \Poset(\mathfrak F)$,
 there is a $p<q$ such that $p$ decides $\dot h\restriction Y$. 
\end{lemma}

\begin{proof} Let $K$ be the set of $k$ such that $q$ does not 
already force a value on  $\dot h(y_k)$ and let $Y' = \{ y_k : k\in
K\}$. Now, using the $V$-genericity of $\mathcal F$, 
 choose $q'<q$ in $\mathfrak F$ so that
 $q'$ forces a value on $F(Y')$. For each $k\in K$, 
choose  $j_k \in F(Y')\cap [n_k, n_{k+1})$ if it is  non-empty,
 otherwise set $j_k = 0$. Assume that the set $K'$, those $k\in K$
such that $q'$ does not force $\dot h(y_k) = j_k$, is infinite. 
It then follows from Lemma \ref{preserve2} that
there is a condition $p<q'$ for which there are
 infinitely many $k\in K'$ such that $y_k\in \dom(h_p)$ and 
$p\Vdash  \dot h(y_k) \neq j_k$. But now we have that $p$ forces that  
 $F(Y' \cap \dom(h_p)) =^* h_p(Y'\cap \dom(h_p)) $ is not almost
equal to $\dot h(Y'\cap \dom(h_p))$, since for each $k\in K'$ with
$y_k\in \dom(h_p)$, 
$j_k \in F(Y')\cap F(\dom(h_p))\setminus \dot h(Y'\cap \dom(h_p))$. 
\end{proof}

\def\Orb{\mathop{\rm Orb}}

\begin{definition} For each condition $q\in \mathbb P(\mathfrak F)$,
  and each $i\in \mathbb N$,
 let $\Orb_q(i) = \{ j  : (\exists p<q)~
p\forces{\mathbb P(\mathfrak F)}{\dot h(i) = j}\}$. Also let $S(k,q) =
\{ s\in \mathcal H_k : q\restriction H_k \subset s
\}$.
\end{definition}

\begin{corollary}
 Our condition $\bar q $ also\label{bounding}
 satisfies
  that for each $i\in \mathbb N $ 
and $q<\bar q $, if 
$\Orb_q(i)$ has more than one element, there is a $k$ 
such that 
 $\{i\}\cup \Orb_{q}(i)\subset   [n_k,n_{k+1})$.
\end{corollary}

\begin{lemma} 
For each $\alpha\in \omega_1$,
 there are $r_\alpha<\bar q \in \mathbb P(\mathfrak
F)$ and $W_\alpha\subset a_\alpha$ such that $r_\alpha\forces{\mathbb P(\mathfrak
  F)}{W_\alpha\in \mathcal W_\alpha \mbox{\ and\ }
\dot h[W_\alpha] \subset b_\alpha}$
\end{lemma}

\begin{proof} Otherwise we can choose a fusion sequence $\{ r^\alpha_k : k\in
 \omega\}$, an increasing sequence of integers $\ell_k$ and values
 $y_k\in a_\alpha\cap [n_{\ell_k},n_{\ell_k+1})$,  and conditions 
$s_k\in \mathcal H_k$ such that $s_k\sqcup r^\alpha_k \forces
{\mathbb P(\mathfrak F)}{\dot h(y_k)\notin b_\alpha}$. There is an
an infinite set $L\subset\omega$  and an $r$
 below  $s_k\sqcup r^\alpha_k$ for all $k\in L$ and such that
$Y=\{ y_k : k\in L\}\subset \dom(h_r)$. 
Since $r$ forces that
$\dot h$  extends
 $h_r$, we have our contradiction since $h_r[Y]\subset^* \dot
 F(a_\alpha) = b_\alpha$ while $r\forces{}{\dot h[Y]\cap b_\alpha =
 \emptyset }$.
\end{proof}

By strengthening the condition $\bar q$, we can assume
that we have the following property.

\begin{lemma} 
For\label{aalpha}
 each integer $\ell$ and each condition $q<\bar q$, there is a
  condition $p<q$ and a set $I\in [\omega_1]^\ell$ such that $p<r_\alpha$
  for each $\alpha\in I$.
\end{lemma}

\begin{proof} Since $\mathbb P(\mathfrak F)$ is ccc,
 there is a condition $q\leq \bar q$ forcing that there
 is an uncountable set
 of $\alpha\in\omega_1$ such that $r_\alpha$ is in the
 generic filter.
\end{proof}

\begin{lemma} For each $\alpha\in \omega_1$, there\label{balpha}
 is an integer
  $\ell_\alpha$ such that for each $k$ and each
$s_k\in S(k,r_\alpha)$, if $|H_k\setminus \dom(s_k)| > \ell_\alpha$,
then $s_k\sqcup r_\alpha$ does not decide $\dot h\restriction W_\alpha
 \cap [n_k,n_{k+1})$.
\end{lemma}

\begin{proof} If such an integer $\ell_\alpha$ did not exist, then we
  could find an infinite $K\subset \omega$ and a sequence
$\langle s_k : k\in \omega\rangle \in \Pi_{k\in \omega} S(k,r)$
with 
$\{ |H_k\setminus \dom(s_k)| : k\in K\}$ diverging to
  infinity, and
 such that $s_k\sqcup r$ decides $\dot h\restriction W_\alpha \cap [n_k,n_{k+1})$
   for each $k$. But then of course, 
$q = \bigcup_{k\in K} s_k \sqcup r_\alpha$ would
  force that $\dot h\restriction W_\alpha = h_\alpha$ for some
  $h_\alpha \in V$. It follows easily 
  from the assumption that $F$ is not trivial at $\mathcal W_\alpha$,
  that there is some $q'<q$ and
  some infinite $W\subset
  W_\alpha$ such that $q'\forces{\mathbb P}{\dot F(W)\cap h_\alpha[W] 
=\emptyset}$.
  By further extending  $q'$ and possibly shrinking $W$,
  we can assume that $W\subset \dom(h_{q'})$. 
This contradicts that $h_{q'}[W]$ is supposed to be forced by $q'$ to be
(mod finite) equal to  both $\dot   h[W]$ and $\dot F(W)$.
\end{proof}

By passing to an uncountable subcollection we may suppose that there
is some $\bar\ell$ such that $\ell_\alpha = \bar\ell$ for all $\alpha$. 
Now define $S'(k,q) = \{ s\in S(k,q) : |H_k\setminus \dom(s) | >
\bar\ell\}$.

\begin{lemma} There\label{orbits}
 is a condition $r$ and an infinite set $K$ such that
$\{ | H_k\setminus \dom(r)| : k\in K\}$ diverges to infinity,
and, for each $k\in K$, we can select $\{ i_s : s\in S'(k,r) \}\subset
[n_k,n_{k+1})$ 
 such that $\Orb_r(i_s ) \cap  \Orb_r(i_{s'})$ is empty
 for each $s\neq s'\in S'(k,r)$,
and $s\sqcup r$ does not decide $\dot h(i_s)$.
\end{lemma}

\begin{proof} Fix any integer $\ell$ and let $L$ be  bigger than
  $(\ell+2)^{\ell+2}$. 
Apply Lemma \ref{aalpha} to find an $r$ which is below
$r_\alpha$ for each $\alpha \in I$ for some $I\subset
  \omega_1$ of cardinality at least $L$. 
    For each $\alpha\in I$, 
we can assume that  $r$ decides 
 $\dot F(W_\alpha) = F(W_\alpha)$ and, by
 Lemma \ref{preserve2}, that 
$\dom(h_r)$ contains $[n_k,n_{k+1})$ for each $k$
  such that $r$ decides $\dot h\restriction W_\alpha\cap 
[n_k,n_{k+1})$ . 
Recall that $a_\alpha$ and $b_\alpha$ were defined in  Lemma \ref{notcccoverfin}.
 We may choose 
 an $m$ such that $[a_\alpha\cup F(W_\alpha)\cup b_\alpha]\cap 
[a_\beta \cup  F(W_\beta)\cup b_\beta]
  \subset m$ for each $\alpha\neq \beta\in I$.

Let $K = \{ k : |H_k \setminus \dom(r)| >\ell\}$. 
It follows from Lemma \ref{balpha}, that for each 
 $\alpha\in I$,  $k\in K$, and $s\in S'(k,r)$,
there is an  $i\in W_\alpha\cap [n_k,n_{k+1})$ for which
$s\sqcup r$ does not decide $\dot h(i)$. 
Therefore, we can select any $k\in K$ and $s_k\in S(k,r)$ with
$\ell\leq |H_k
\setminus \dom(s_k)| < \ell+2$ and fix any injection
 from $ S'(k,s_k\cup r)$ into $I$ (i.e. $\{ \alpha_s : s\in 
S'(k,s_k\cup r)\}$). For each $s\in S'(k,s_k\sqcup r)$, there is an
$i_s\in W_\alpha\cap [n_k,n_{k+1})$ such that $s \sqcup r$ does not
  decide $\dot h(i_s)$. Since $r$ forces that 
$\{i_s,\dot h(i_s)\}\subset
 a_{\alpha_s}\cup b_{\alpha_s}$ and for $s'\neq s$, $r$ forces that 
$i_{s'},\dot
  h(i_{s'})\notin a_{\alpha_s}\cup 
b_{\alpha_s}$,  we have satisfied the requirement
  that $i_s \notin \Orb_r(i_{s'})$ (the hard part was making them
  distinct). To complete the proof of the Lemma, simply perform
  another fusion to inductively choose $s_k$ and  extend $s_k \sqcup r$
  for larger and   larger $\ell$.  
\end{proof}

\begin{theorem} The trivial ideal, $\triv(F)$, is ccc\label{repeatcccfin} over fin.
\end{theorem}

\begin{proof} 
Let $r$ and the sequence $\{ \{i_s : s\in S'(k,r) \} : k\in K\}$ be as
constructed in Lemma \ref{orbits}. Since $\{ | H_k \setminus \dom(r) |
: k\in K\}$ diverges to infinity, we may assume that 
$\dom(r) \supset [n_k,n_{k+1})$ for each $k\notin K$.  We define a
  $\mathbb P$-name of an ultrafilter. Each $q<r$ forces that the
set $X(q)=\bigcup_{k\in K} \{ i_s : s\in S'(k,q) \} $ is a member. 
Let $x$ be any ultrafilter extending this filter. 
We claim that $\dot \Phi(x)$ has no value. Assume first that there is
some $q<r$ which forces that $X(r)\in \dot \Phi(x)$ (i.e. $q$ forces
that $ F^{-1}(X(r))\in x$). We may then further assume that some 
$q'<q$
forces that $X(q)$ has some infinite subset $Y(q)\in x$ such that 
for some $m$,
$(F(Y(q))\setminus m) \subset X(r)$. Of course, $Y(q)$ is just a
set in $V$ and $Y(q)\cap X(q')$ is large. Choose an infinite sequence
$\{ k_j : j\in \omega\}$ so that for each $j$, we can choose
$i_j  \in Y(q)\cap \{ i_s : s\in S'(k_j,q')\}$
and so that 
$\{ | H_{k} \setminus \dom(q')| : k\notin \{ k_j : j\in \omega\}$
 diverges to infinity. For each $j$, let $s_{k_j} $ denote the 
member $s$ of $S'(k_j,q')$ such that $i_j = i_s$. For each $j$, choose
any $s'_j\in S'(k_j,q')$ which extends $s_{k_j}$ that satisfies
that 
$s'_j\sqcup q' \forces{\mathbb P(\mathfrak F)}{\dot h(i_j) \neq  i_j}$. 
 Let $h^*$ denote the function with domain $\{ i_j : j \in
\omega\}$ satisfying that $
s'_j\sqcup q' \forces{\mathbb P(\mathfrak F)}{\dot h(i_j) =
  h^*(i_j)}$. Note that 
$h^*(i_j)\notin X(r)$ for all $j$ since $i_j$
is the only member of $\Orb_r(i_j)$ in $X(r)$. 
 We can extend the condition $q'\cup \bigcup_j
 s'_{j}$ further to some $q^*$ so that $\dom(q^*)\supset
 [n_{k_j},n_{k_j+1})$ for each $j$. 
We observe that $J=
\{ i_{j} : j\in \omega\}\subset Y(q)\cap \dom
(h_{q^*})$. By removing finitely many elements, we may assume that
$q^*$ forces that $ h^*$ agrees with $h_{q^*}$ on $J$. However, we now
have a contradiction since $h_{q^*}[J] =^* F(J) \subset^* 
F(Y(q))\subset^* X(r)$.

Now assume that there is a $q<r$ which forces that $X(r)$ is not in
$\dot \Phi(x)$. There is a $q'<q$ and a $Y(q)\subset X(q)$ 
such
that for some $Z\subset \Naturals\setminus X(r)$,
$q'$ forces that $Y(q)\in x$ and $Z=  F(Y(q))$.
Again select a sequence $\{ k_j : j\in \omega\}$ so that 
for each $j$,  $[n_{k_j}, n_{k_j+1})\cap Y(q)\cap X(q')$ is not empty, 
and choose $i_j$ from this set.
We may choose this sequence so that
 $\{ | H_k \setminus \dom(q')| : k\notin \{ k_j : j\in \omega\}$
diverges to infinity. 
For each $j$, let $s_j\in S'(k_j,q')$ be chosen
so that $i_j=i_{s_j}$. If for infinitely many $j$,
it is possible to select $s'_j\in
S'(k_j,s_j\cup q')$ so that $s'_j\cup q'$ forces that $\dot h(i_j)$ is
not in $Z$, then we select such an $s'_j$. The proof then proceeds
much as in the first case because it will allow us to obtain that
for an infinite $J$,
$\dot h[\{ i_j : j\in J\}]$ is disjoint from $Z$ in contradiction to 
 $F(\{ i_j : j\in J\})$ being contained in $Z$.  In the other case, we
select $z_j\in Z$ such that $s_j$ has an extension forcing that $\dot
h(i_j)$ is equal to $z_j$, but we also know that we can (and do)
select $s'_j$ extending $s_j$ to force that $\dot h(i_j)$ is not equal
to $z_j$. Applying the same arguments to the automorphism $\dot
F^{-1}$ we may certainly select an infinite $J\subset \omega$ and a
$q^* < q'\cup \bigcup_j s'_j$ so that $\{ z_j : j\in J\}$ is in the
range of $h_{q^*}$. It follows that there is a sequence $Y=\{ y_j : j\in
J\}\subset Y(q)$ such that $h_{q^*}(y_j) = z_j$ for each $j\in J$. 
Clearly then this puts $z_j\in \Orb_r(y_j)$ for all but finitely many
$j$. By Lemma \ref{bounding}, we actually have that $y_j$ is also from 
$[n_{k_j},n_{k_j+1})$ and so the contradiction is that 
we have arranged that $\Orb_r(y_j)$ and $\Orb_r(i_j)$ are disjoint.
\end{proof}

\section{all homeomorphisms are trivial when forcing with $\Poset_0$ }

In this section we first restrict to the case where  $\Poset$ is  $\Poset_0$ 
and we  continue with
  the assumptions on $F$ 
 and $\dot h$ and
 the sequence $\{n_k : k\in \omega\}$ established in
 the previous section.
  We begin with a $\bar q $ and the sequence $\{ n_k, m_k : k\in \omega\}$
as developed in the paragraph following  Lemma \ref{preserve2}.
In particular, that $[2^{m_k},2^{m_k+1})$ contains
$[n_k,n_{k+1})\setminus \dom(\bar q )$.
For each $k$,  we will now let $H_k$ be the set $
[2^{m_k},2^{m_k+1})\setminus \dom(\bar q ) $.
We may assume that $\{ |H_k| : k\in \omega\}$ diverges to infinity.
Say that a condition $q$ is standard, if for each $\ell>0$, there are
at most finitely many $k$ such that $H_k\setminus \dom(q)$
has cardinality $\ell$. The standard conditions are dense below $\bar q $
in $\Poset_0$. For a standard condition $q$, let $K(q)$ denote those
 $k$ such that $H_k\setminus \dom(q)$ is not empty. It follows then
 that $\{ | H_k \setminus \dom(q) | : k\in K(q)\}$ diverges to infinity. 
Recall that $q$ is identically $0$ on $\dom(q)\cap S^k$ 
for all $k\in K(q)$.  For a condition $p$ and $i\in H_k\setminus \dom(p)$,
we abuse notation and suppose that $p\cup \{(i,1)\}$ denotes
the smallest condtion in $\Poset_0$ that contains
 $p\cup \{(i,1)\}$.

\begin{lemma} If   $p_0<\bar q$ in $\Poset_0$ 
is a\label{p0trivial}
 standard condition such
that no extension of $p_0$ decides $\dot h(t)$ for all values of $t$, 
then there is an extension $p< p_0$ such that for all $i\in \mathbb N\setminus
\dom(p)$, there is a value $t_i$ so that for some distinct pair $u_i,v_i$,
$p\cup \{(i,0)\}\Vdash \dot h(t_i) = u_i$ and $p\cup \{(i,1)\}\Vdash
\dot h(t_i) = v_i$. 
\end{lemma}

\begin{proof}
We proceed by a simple recursion. By induction on $\ell$, suppose we
have chosen $p_\ell$  together with a family 
$\{ i(k,j) : j< \ell\} \subset H_k\setminus \dom(p_\ell)$ for all
$k\in K(p_\ell)$. 
 We assume  that for each $j< \ell$ and $k\in K(p_\ell)$, there is a value
  $t_{k,j}$ so that $p_{j+1}\cup \{ (i(k,j'), 0) : j'\leq j\} $ 
  and $p_{j+1}\cup \{ (i(k,j'), 0) : j' < j\}\cup \{ ( i(k,j), 1)\} $
  force distinct values, 
   $u_{k,j}, v_{k,j}$, on $\dot h(t_{k,j})$.   
As usual in such a fusion, we assume that $p_{j+1} \restriction
n_{m_j} \subset p_j$ 
so that we will have that $\bigcup_{\ell} p_\ell$ is a condition.
   Now we may choose a sequence
    $\langle t_{k,\ell} : k\in K\rangle$ (for some infinite $K\subset K(p_\ell)
    $) 
    such that, for each $k\in K$, $t_{k,\ell} \in [n_k, n_{k+1})$ and
     $p_\ell\cup \{ (i(k,j) ,0) : j<\ell\}$ does not force a value on $\dot h(t_{k,\ell})$. 
For each $k\in K$, there are two values $\bar \i^k_0, \bar \i^k_1$     from
 $H_k\setminus \left(\dom(p_\ell)\cup \{ i(k,j) : j< \ell \}\right)$, 
such that $p_\ell\cup \{(\bar\i^k_0,1)\}$ and $p_\ell\cup \{(\bar \i^k_1,1)\}$ force
distinct values, $v^k_0,v^k_1$, on $\dot h(t_{k,\ell})$. 
Using Lemma \ref{grabY}, choose $\bar p_{\ell+1} < p_\ell$ such
that 
for all $k\in K(\bar  p_{\ell+1}) \subset K$,
$\{ i(k,j) : j< \ell\}$ is disjoint from $\dom(\bar p_{\ell+1})$ 
 and $\bar p_{\ell+1}\cup \{ (i(k,j) ,0) : j<\ell\}$ forces a value, $u_{k,\ell}$, 
 on  $\dot h(t_{k,\ell})$. Suppose, without loss of generality, that
 $v^k_1 \neq u_{k,\ell}$ 
 and let $i_{k,\ell} = \bar\i^k_1$. It follows that $i_{k,\ell}\in
 \dom(\bar p_{\ell+1})$  
 and so define $p_{\ell+1}$ to be the condition we get by removing  $i_{k,\ell}$
 from the domain of $\bar p_{\ell+1}$ for all $k\in K=K(p_{\ell+1}) $.
 
 When the recursion is finished, we choose any increasing sequence 
  $\{ k_\ell : \ell \in \omega\}$ so that $k_\ell \in K(p_{\ell+1})$, and $p<p_0$
  any condition so that $K(p) = \{ k_\ell : \ell \in \omega\}$, and $H_{k_\ell} \setminus
  \dom(p)
  =
     \{ i(k_\ell,j) : j< \ell\}$. Of course this implies that $p$ is constantly 0
      on each $H_{k_\ell}\cap \dom(p)$. For each $k=k_\ell$ and $j< \ell$, 
        we have that $p\cup \{ (i(k,j), 1)\}$ forces the value $v_{k,j}$ on $\dot h(t_{k,j})$
    because of Lemma \ref{preserve2}.    
    And similarly, since $p_{\ell+1}\restriction H_{k}  \subset p$, we have
    that $p\cup \{ (i(k,j'), 0) : j'\leq j\}$ forces that $\dot h(t_{k,j}) = u_{k,j}$. 
Because of this, we have that if $q< p$ is such that $k\in K(q)$ and $q$ forces a value
  on $\dot h(t_{k,j})$, then this value has to be $u_{k,j}$. We finish the construction by
   another more routine recursion. There should be no risk of confusion
   if we re-use the notation $p_1,p_2$ etc. for the values in this new recursion.   
    For each $k\in K(p)$, let $j_{k,0}$ denote the largest value
   so that $i(k,j_{k,0})\notin \dom(p)$.  By Lemma \ref{grabY}, there
   is  a condition $p_1<p$ so that $p_1$ forces
   a value on $\dot h\restriction \{ t_{k,j_{k,0}} : k\in K(p)\}$. Again, as discussed above, 
   we have that $p_1$ forces that $\dot h(t_{k,j_{k,0}}) = u_{k,j_{k,0}}$ for each $k\in K(p_1)$.
   There is an infinite set $K_1\subset K(p_1)$ such that there is a largest $j_{k,1} < j_{k,0}$ such that $i(k,j_{k,1})\notin \dom(p_1)$.
    Find a condition $p_2<p_1$ which forces a value on $\dot h\restriction 
     \{ t_{k,j_{k,1}}  : k \in K_1\}$. Continue this induction. Again there is a sequence 
     $\{ k_\ell : \ell \in \omega \}$ such that $j_{k,\ell}$ was successfully chosen for $k=k_{\ell+1}$. 
     We extend $p$ to a condition $p'$ so that $K(p') = 
     \{ k_\ell : \ell\in \omega\}$ and $H_{k_\ell} \setminus
      \dom(p')$ is equal to $\{ i(k_\ell, j_{k_\ell,m}) : m < \ell\}$. We still have that
       $p'\cup \{(i,1)\} \Vdash\dot h(t_{k,i}) = v_{k,i}$ for each $i\in H_k\setminus \dom(p')$, but
       we can now show that $p'\cup \{ (i,0)\}$ forces that $\dot h(t_{k,i})= u_{k,i}$.
       The simplest way to do this is to consider any  $i'\in H_k\setminus 
       \dom(p')$ with $i'\neq i$. If $i'<i$, then the condition $p'\cup \{ (i',1)\}$ is compatible
       with  $p_{m+1}\restriction H_k$ 
       where $m$ is chosen so that $i = i(k,j_{k,m})$. On the other hand
       if $i'>i$, then $p'\cup \{ (i',1)\}$ is compatible with $p\cup \{(i(k,j),0) : j\leq j_{k,m} \}$.
       Since each of these force that $\dot h(t_{k,i}) = u_{k,i}$, we have that $p'\cup \{ (i,0)\}$
       forces that $\dot h(t_{k,i}) = u_{k,i}$.       
\end{proof}

\begin{theorem}
In the extension obtained by forcing over a model of PFA
by\label{repeatnoauto}
   $\Poset_0$ 
all automorphisms on $\pomegaf$ are trivial.  
\end{theorem}

\begin{proof}
  Fix a condition $p$ as in Lemma \ref{p0trivial}
    satisfying that 
 for  each $i\notin \dom(p)$ there is  a $t_i$
such that there are distinct values
 $u_i$,  $v_i$ 
that $p\cup\{(i,0)\}$ and
$p\cup \{(i,1)\}$, respectively, force on $\dot h(t_i)$. 
We prove that this leads to a contradiction.
Choose a condition
$q<p$ and a set $Y$ so that $q$ forces that 
 $F(Y) = \{ u_i : i\notin \dom(p)\}$.
 Let $L_0 = \{ i\notin \dom(q) : t_i \notin Y\}$.
If $L_0$ is infinite, then we have a contradiction by 
choosing any 
infinite subset $L'$ of $L_0$ so 
that $L'\cap H_k$ has at most one element for each $k$, 
and considering 
the condition $q'= q\cup \{ (i,0) : i \in L' \}$. 
We now have that $q'$ forces that   $F(\{t_i : i\in  L' \}) =^* h_{q'}(\{t_i : i\in L'\})$ 
will be almost contained in  $F(Y)$ while $\{t_i : i\in L' \}$
 is disjoint from $Y$.

Now suppose that $L_0$ is finite. If $\{ v_i : i\notin \dom(q) \}
\setminus F(Y)$ is infinite, then we choose an infinite $L'$
so that $L'\cap H_k$ is empty for infinitely many $k\in K(q)$
 and so that $\{ v_i : i\in L'\}$ is disjoint from $F(Y)$.
Again the extension $q' = q \cup \{ (i,1) : i\in L'\}$  will force
that $F(\{t_i : i\in L'\}) =^* h_{q'} (\{t_i : i\in L'\})$, but this
contradicts that it is supposed
to be mod finite contained in $F(Y)$. 

 The final case then is that there is
an infinite sequence $L'\subset K(q)$ such
that $K(q)\setminus L'$ is still infinite and there is a sequence 
of pairs $\{ i_k , i_k' : k\in L'\}$
   such that 
 $i_k, i_k'$ are distinct members of  $H_k\setminus \dom(q)$ and 
 $v_{i_k}   = u_{i_k'} $ for each $k\in L'$. 
 Now we have that the extension $q' 
= q\cup \{ (i_k,1)  : k\in L'\}
\supset q\cup \{ (i_k,1), (i'_k,0) : k\in L'\}$  will force that $\dot h$ 
is not 1-to-1.
\end{proof}

\section{More properties of the poset $\Poset_1$}

\begin{theorem} In a model obtained by forcing with the\label{ramsey}
 poset $\Poset_1$ over a model of PFA,
there is a non-trivial autohomeomorphism $\varphi$ of
$\mathbb N^*$ and two regular closed copies $A,B$ of
$\mathbb N^*$ and a tie-point $\mathcal W$ such that 
\begin{enumerate}
\item $\varphi[A]=B$ and $\varphi[B] =A$, and $A\cap B=\{\mathcal W\}$,
\item $\mathcal W$ is the only point on the boundary of each of $A$ and $B$,
\item $\varphi$ is the identity on $\mathbb N^*\setminus (A\cup B)$.
\end{enumerate}
\end{theorem}

\begin{proof}
We will define a strange sequence, $\{ \dot t_m : m\in \omega\}$, of 
$\Poset_1$-names of pairs. These will code liftings of the maps
between $A$ and $B$ (each will ``pick'' a point from the pair) and the
mappings of each onto $\mathbb N^*$ (each member from the $m$-th pair
being sent to $m$).  The difficult part of the construction is to ensure that $A$
and $B$ meet in a single ultrafilter. 

 For each $m\in \omega$ and each function $\sigma\in 2^{[2^m,2^{m+1})}$, we
 will choose a pair $a_\sigma \subset [2^m, 2^{m+1})$. The definition of
 $\dot t_m$ will simply be that a condition $p\in \Poset_1$ such that 
$[2^m,2^{m+1})\subset \dom(p)$, will force that $\dot t_m$ is equal to
  $a_{p\restriction [2^m,2^{m+1})}$. Analogous to the definition of
    $K(p)$ above, let  $M(p)$ denote the set $
 \{ m\in\omega :
    [2^m,2^{m+1})\not\subset \dom(p)\}$ for each $p\in
    \Poset_1$. Without mention, we will assume that we work with the
    dense set of conditions which satisfy that $\{
    |[2^{m},2^{m+1})\setminus \dom(p) | : m\in M(p)\}$ diverges to
    infinity. 

For each $p\in \Poset_1$, let $T(p) = \{ t_m : m\notin M(p) \mbox{\
  and\ } p\Vdash t_m = \dot t_m\}$, $A(p) = \{ \min(t_m) : t_m\in
T(p)\}$ 
and $B(p) = \{ \max(t_m) : t_m \in T(p)\}$. If $G$ is a generic filter
on $\Poset_1$, then we will set $A$ to be the closure of the open set
 $\bigcup \{ A(p)^* : p\in G\}$ and $B$ to be the closure of the open
 set
 $\bigcup\{ B(p)^* : p\in G\}$. 

For each $p\in \Poset_1$, let $W(p)$ equal
$\bigcup_{m\in M(p)} \{ a_\sigma : \sigma\cup p \in \Poset_1\}$.
 Now define the filter $\mathcal W$ to
 be the filter generated by the 
family $\{ W(p)  : p\in G \}$. 
If we define these names so that  $\mathcal W$ is forced to be an
ultrafilter, then it is quite routine to check that $\mathcal W$ is
the only boundary point of each of $A$ and $B$ and that the map
sending $\mathbb N^*$ onto $\mathbb N^*$ obtained by extending the map
sending each interval $[2^m,2^{m+1})$ to $\{ m\}$ will restrict to a
homeomorphism on each of $A$ and $B$. 
This implies that there is a homeomorphisms from $A$ onto $B$
that sends $\mathcal W$ to itself. It should then be clear
that there is a homeomorphism $\varphi$ as in the statement
of the theorem. 

Now we set about showing that there is such a sequence of names. We
will define, for $m\in \omega$ and $\sigma \in 2^{[2^m,2^{m+1})}$, the
value of $a_\sigma$ based only on the cardinality of
$\sigma^{-1}(1)$. In fact some Ramsey theory says this will
effectively be the
case anyway.  To make these choices we now introduce the idea of an
$\ell$-structure for $\ell \in \omega$. 

The $0$-structure will be the
empty set. We let $L_0=2$ and $n_0=6 = L_0+ L_0^{L_0}$. 
We next define  a family of pairs
$\{ a_{\langle i\rangle} : i < n_0\}$:
$$\begin{array}{ccc}
a_{\langle 0\rangle} = \{0,1\},& a_{\langle 1\rangle} = \{2,3\}, &
a_{\langle 2\rangle}= \{0,2\},\\
 a_{\langle 3\rangle} = \{ 0,3\},&
a_{\langle 4\rangle}= \{1,2\},& a_{\langle 5\rangle} = \{ 1,3\}
\end{array}
$$
This assignment satisfies that for each set $Y$ such that
 $Y\cap a_{\langle i\rangle}$ is not empty for each $i<L_0$, there
is a $j<n_0$ such that $Y\supset a_{\langle j\rangle}$. This is the process by
which we will ensure that the above defined $\mathcal W$ is an
ultrafilter. 

By recursion on $\ell$, we define 
  $L_{\ell} = 2\,n_0\,n_1 \cdots n_{\ell-1} $, set
 $n_{\ell} = L_\ell + L_\ell^{L_\ell}$, and we define
 our $\ell$-structure based on the cartesian product
$$\mathcal N_\ell = 
n_{\ell}\times n_{\ell-1}\times\cdots \times n_1\times n_0~~.$$ 
It is awkward, but ultimately more convenient, to have this product in
descending order.  For each $j< \ell$, also let 
$\mathcal N_{\ell,j} = n_{\ell}\times \cdots \times n_{j}$.

 An $\ell$-structure is a family 
$\langle \{ a_x : x\in \mathcal N_\ell \}, 
 \{ Y_\rho : \rho \in \bigcup_{j<\ell} \mathcal N_{\ell,j}\}\rangle $ 
satisfying
\begin{enumerate}   
\item for each $x\in \mathcal N_\ell$, $a_x$ is a pair of integers
\item for each $\rho\in \bigcup_{j<\ell}\mathcal N_{\ell,j}$, \ 
 $Y_\rho$ is the union of all
  $a_x$ with $x\in  \mathcal N_\ell$ and $\rho\subset x$,
\item for each $j<\ell-1$ and  $\rho\in \mathcal N_{\ell,j}$, the
  family
   $\langle \{ a_x : \rho\subset x\in \mathcal N_{\ell}\}, 
  \{ Y_\psi : \rho\subset \psi \in \bigcup_{i<\ell} \mathcal
  N_{\ell,i}\}\rangle$ is an $(\ell-j)$-structure (with a confusing
  prefix of $\rho$ on each index),
\item the family $\{ Y_{\langle m\rangle } : m<L_\ell \}$ are pairwise
  disjoint and $Y_\emptyset$ is the union,
\item for each $Y\subset Y_\emptyset$ such that $Y\cap Y_{\langle
    m\rangle}\neq\emptyset $ for each $m< L_\ell$, there is a
  $k<n_\ell$
 such that $Y\supset Y_{\langle k\rangle}$. 
\end{enumerate}

The construction is quite straightforward. Let $\{\bar a_x : x\in
\mathcal N_{\ell{-}1} \}$ be the pairs from an $\ell{-}1$-structure. The
definition of $L_\ell$ ensures that it exceeds 
 the cardinality of, $\bar Y_\emptyset$, the union of these pairs. 
Let $\{ Y_{\langle m\rangle} : m < L_\ell\}$ be a pairwise disjoint
family of sets of integers each of the same cardinality as
 $\bar Y_\emptyset$. Similarly, let $\{ Y_{\langle k\rangle} : m\leq
 k<n_\ell\}$ be a family of sets, each of cardinality $|\bar
 Y_\emptyset|$, so that the last inductive assumption is satisfied. 
For each $k<n_\ell$, fix a bijection, $f_k$, 
 between $\bar Y_\emptyset$ and
$Y_{\langle k\rangle}$ and define $a_{k^\frown \bar x} = f_k[ \bar a_{\bar
  x}]$ for each $\bar x \in \mathcal N_{\ell{-}1}$.

For each $\ell$, let $c_\ell$ denote the order-preserving mapping from
$\mathcal N_\ell$ with the natural lexicographic ordering into an
initial segment of $[0,L_{\ell+1})$. It is important to observe that
for each $\rho\in \bigcup_{j<\ell} \mathcal N_{\ell,j}$, the set
 $[\rho] = \{ x\in \mathcal N_{\ell} : \rho\subset x\}$ is an interval
 in the lexicographic ordering, and so, $c_\ell([\rho])$ is an
 interval in $[0,L_{\ell+1})$.  We may also assume that we have, 
 for each $\ell$, a fixed $\ell$-structure so that the set
 $Y^\ell_\emptyset = Y_\emptyset$ from this structure is an initial
 segment of integers.  

For each integer $m$, choose $\ell=\ell_m$ maximal 
so that $L_{\ell_m+1}\leq 2^m$ (hence the interval $[2^m,2^{m+1})$
will support an $\ell$-structure). For each $\sigma\in
2^{[2^m,2^{m+1})}$ such that there is an $x\in \mathcal N_{\ell}$ with 
$c_\ell(x) =|\sigma^{-1}(1)| $, define $a_\sigma$ to be the pair obtained
by adding $2^m$ to each member of $a_x$. It follows that
$a_\sigma\subset [2^m, 2^{m+1})$. If there is no such $x\in \mathcal
N_\ell$,  let $a_\sigma = \{ 2^m, 2^m{+}1\}$. 

Define the condition $p_0\in \Poset_1$ by the prescription that for
all $m$ and $\ell = \ell_m$, 
 $p_0\restriction[2^m, 2^{m+1})$ is the partial function which
is $0$ on the segment $[2^m+|Y^\ell_\emptyset|, 2^{m+1})$. This
ensures that for all $\sigma\in 2^{[2^m,2^{m+1})}$ which extend
 $p_0\restriction [2^m,2^{m+1})$, there will be an $x\in \mathcal
 N_{\ell}$ such that $c_\ell(x) = |\sigma^{-1}(1)|$. 

To finish the proof, we prove that $p_0$ forces that $\mathcal W$ is
an ultrafilter. That is, if $q<p_0$ and $Y\subset \mathbb N$, then
there is a $p<q$ such that $Y$ either contains, or is disjoint from,
 $W(p)$. 

We may assume that the sequence $\{ k_m = 
|[2^m,2^{m+1})\setminus \dom(q) | : m \in M(q)\}$ diverges to
infinity. For each $m\in M(q)$, let $q_m = q\restriction
 [2^m,2^{m+1})$. 
Also, for $m\in M(q)$, let $i_m$ be the largest integer so
that $n_{i_m} < k_m/3$. By thinning out further (using any extension
of $q$), we can also assume
that  $\{ i_m : m\in M(q)\}$ diverges to infinity. It then follows
that for each $m\in M(q)$ there is a $\rho_m\in \mathcal
N_{\ell_m,i_m}$ such that the interval $c_{\ell_m}[\rho_m]$ is
contained in $[|q_m^{-1}(1)| , |q_m^{-1}(1)| +k_m)$. By inductive
hypotheses (3) and (5) in the definition of an $\ell_m$-structure, 
 there is an extension $\psi_m$ of $\rho_m$ with $\psi_m\in
\mathcal N_{\ell_m, i_m+1}$ such that $Y$ either 
contains, or is disjoint from, $W_m = 2^m+ Y^{\ell_m}_{\psi_m}$ (i.e. 
$W_m = \bigcup\{ a_\sigma : \sigma \in 2^{[2^m,2^{m+1})} \ \mbox{and}\ 
 |\sigma^{-1}(1)| \in c_{\ell_m}[\psi_m]\}$). By symmetry, we may
 assume that there is an infinite $K\subset M(q)$ such that 
  $Y$ contains $W_m $ for all $m\in K$. Let $p<q$ be
any  condition such that $p^{-1}(0) = q^{-1}(0)$ (no more 0's are
added) 
and  for each $m\in K$, 
 the minimum element of $c_{\ell_m}[\psi_m]$ is the number of values
 in $[2^m,2^{m+1})$ which are sent to 1 by $p$. Notice that for
 $m\in K$, $|[2^m,2^{m+1})\setminus \dom(p)| > n_{i_m+2}$ and so we
may assume that $M(p)=K$. 
In other words, $p$ 
will satisfy that
  $W(p) = \bigcup_{m\in M(p)}W_m$. This shows that $p$ forces that 
 $Y$ contains a member of $\mathcal W$. 
\end{proof}

\let\MR\relax


\begin{thebibliography}{1}


\bibitem{DowPAMS}
Alan Dow.
\newblock A non-trivial copy of {$\beta\mathbb{N}\setminus \mathbb{N}$}.
\newblock {\em Proc. Amer. Math. Soc.}, 142(8):2907--2913, 2014.
		
 
\bibitem{Dowomega1}
Alan Dow, \emph{P{FA} and {$\omega^*_1$}}, vol.~28, 1988, Special issue on
  set-theoretic topology, pp.~127--140. \MR{932977}

\bibitem{DSh1}
Alan Dow and Saharon Shelah.
\newblock Tie-points and fixed-points in {$\Bbb N^*$}.
\newblock {\em Topology Appl.}, 155(15):1661--1671, 2008.


\bibitem{DSh2}
Alan Dow and Saharon Shelah.
\newblock More on tie-points and homeomorphism in {$\Bbb N^\ast$}.
\newblock {\em Fund. Math.}, 203(3):191--210, 2009.


\bibitem{DoSiVa}
Alan Dow, Petr Simon, and Jerry~E. Vaughan, \emph{Strong homology and the
  proper forcing axiom}, Proc. Amer. Math. Soc. \textbf{106} (1989), no.~3,
  821--828. \MR{MR961403 (90a:55019)}

\bibitem{DTech1}
Alan Dow and Geta Techanie, \emph{Two-to-one continuous images of 
{$\mathbb N\sp
  *$}}, Fund. Math. \textbf{186} (2005), no.~2, 177--192. \MR{MR2162384
  (2006f:54003)}


\bibitem{Farah00}
Ilijas Farah, \emph{Analytic quotients: theory of liftings for quotients over
  analytic ideals on the integers}, Mem. Amer. Math. Soc. \textbf{148} (2000),
  no.~702, xvi+177. \MR{MR1711328 (2001c:03076)}

%\bibitem{RLevy}
%Ronnie Levy, \emph{The weight of certain images of {$\omega$}}, Topology Appl.
  %\textbf{153} (2006), no.~13, 2272--2277. \MR{MR2238730 (2007e:54034)}


\bibitem{shel.pf}
S.~Shelah, {\em Proper forcing}, Lecture Notes in Mathematics, vol. 940,
  Springer-Verlag, Berlin, 1982.


\bibitem{ShStpfa}
Saharon Shelah and Juris Stepr{\=a}ns.
\newblock P{FA} implies all automorphisms are trivial.
\newblock {\em Proc. Amer. Math. Soc.}, 104(4):1220--1225, 1988.


%\bibitem{step.15}
%S.~Shelah and J.~Stepr\={a}ns.
%\newblock Non-trivial homeomorphisms of $\beta {N}\setminus {N}$ without the
  %{C}ontinuum {H}ypothesis.
%\newblock {\em Fund. Math.}, 132:135--141, 1989.

\bibitem{step.28}
S.~Shelah and J.~Stepr\={a}ns.
\newblock Somewhere trivial autohomeomorphisms.
\newblock {\em J. London Math. Soc. (2)}, 49:569--580, 1994.


\bibitem{ShSt735}
Saharon Shelah and Juris Stepr{\=a}ns, \emph{Martin's axiom is
  consistent with 
  the existence of nowhere trivial automorphisms}, Proc. Amer. Math. Soc.
  \textbf{130} (2002), no.~7, 2097--2106 (electronic). \MR{1896046
  (2003k:03063)}

\bibitem{step.29}
Juris Stepr{\=a}ns, 
\emph{The autohomeomorphism group of the \v {C}ech-{S}tone
  compactification of the integers}, 
Trans. Amer. Math. Soc. \textbf{355}
  (2003), no.~10, 4223--4240 (electronic). \MR{1990584 (2004e:03087)}


\bibitem{veli.def}
B.~Velickovic.
\newblock Definable automorphisms of ${\mathcal P}(\omega)/\fin$.
\newblock {\em Proc. Amer. Math. Soc.}, 96:130--135, 1986.

\bibitem{veli.oca}
Boban Veli{\v{c}}kovi{\'c}.
\newblock ${\rm {O}{C}{A}}$ and automorphisms of ${\mathcal P}(\omega)/\fin$.
\newblock {\em Topology Appl.}, 49(1):1--13, 1993.

\end{thebibliography}
\end{document}